\documentclass[11pt,a4paper]{amsart}
\usepackage[utf8]{inputenc}
\usepackage[top=3.5cm,bottom=3.5cm,left=3cm,right=3cm]{geometry}
\usepackage{amsmath,amscd,amssymb,amsthm}
\usepackage[english]{babel}
\usepackage{lipsum}
\usepackage{mathtools}
\usepackage{enumerate}
\usepackage{setspace}
\usepackage{mathrsfs}
\usepackage{tikz,qtree}
\usepackage{stmaryrd,empheq,wasysym,bbm}
\usepackage{calc}
\usepackage{listings}
\usepackage{bookmark}
\usepackage[OT2,T1]{fontenc}

\DeclareMathOperator{\gal}{Gal}

\DeclareMathOperator{\aut}{Aut}

\DeclareMathOperator{\id}{id}

\DeclareMathOperator{\St}{St}

\theoremstyle{definition}
\newtheorem{definizione}{Definition}[section]
\newtheorem{rmk}[definizione]{Remark}

\theoremstyle{plain}
\newtheorem{thm}[definizione]{Theorem}
\newtheorem{cor}[definizione]{Corollary}
\newtheorem{lem}[definizione]{Lemma}
\newtheorem{prop}[definizione]{Proposition}

\newtheorem{conj}[definizione]{Conjecture}

\newtheorem*{theorem*}{Theorem}

\newcommand{\bdef}{\begin{definizione}}
\newcommand{\ndef}{\end{definizione}}
\newcommand{\bthm}{\begin{thm}}
\newcommand{\nthm}{\end{thm}}
\newcommand{\brmk}{\begin{rmk}}
\newcommand{\nrmk}{\end{rmk}}
\newcommand{\brmks}{\begin{rmks}}
\newcommand{\nrmks}{\end{rmks}}
\newcommand{\bcor}{\begin{cor}}
\newcommand{\ncor}{\end{cor}}
\newcommand{\bpro}{\begin{proof}}
\newcommand{\npro}{\end{proof}}
\newcommand{\blem}{\begin{lem}}
\newcommand{\nlem}{\end{lem}}
\newcommand{\bprop}{\begin{prop}}
\newcommand{\nprop}{\end{prop}}
\newcommand{\bex}{\begin{ex}}
\newcommand{\nex}{\end{ex}}
\newcommand{\bexs}{\begin{exs}}
\newcommand{\nexs}{\end{exs}}
\newcommand{\disp}{\displaystyle}
\newcommand{\sub}{\subseteq}

\newcommand{\N}{\mathbb{N}}
\newcommand{\Z}{\mathbb{Z}}
\newcommand{\Q}{\mathbb{Q}}

\newcommand{\D}{\mathbb{D}}
\newcommand{\T}{\mathbb{T}}
\renewcommand{\T}{\mathcal{T}}

\newcommand{\p}{\mathfrak{p}}

\renewcommand{\O}{\mathcal{O}}
\newcommand{\X}{\mathcal{X}}

\makeindex

\author[A. Ferraguti]{Andrea Ferraguti}
\address{University of Cambridge\\
DPMMS\\
Centre for Mathematical Sciences\\
Wilbeforce Road, Cambridge, CB3 0WB, UK\\
}
\email{af612@cam.ac.uk}

\date{}
\title{The set of stable primes for polynomial sequences with large Galois group.}

\keywords{Arboreal Galois representations, number fields, stable primes, natural density}
\subjclass[2010]{Primary  11R32, 11R45, 20E08}

\begin{document}
\begin{abstract}
 Let $K$ be a number field with ring of integers $\O_K$, and let $\{f_k\}_{k\in \N}\subseteq \O_K[x]$ be a sequence of monic polynomials such that for every $n\in \N$, the composition $f^{(n)}=f_1\circ f_2\circ\ldots\circ f_n$ is irreducible. In this paper we show that if the size of the Galois group of $f^{(n)}$ is large enough (in a precise sense) as a function of $n$, then the set of primes $\p\sub\O_K$ such that every $f^{(n)}$ is irreducible modulo $\p$ has density zero. Moreover, we prove that the subset of polynomial sequences such that the Galois group of $f^{(n)}$ is large enough has density 1, in an appropriate sense, within the set of all polynomial sequences.
\end{abstract}
\maketitle
\section{Introduction}
In the recent years, there has been a growing interest in the field of arithmetic dynamics (see for example \cite{bened},\cite{bosjon},\cite{fermic1},\cite{fermic3},\cite{micheli},\cite{jon4},\cite{jon2},\cite{jon3},\cite{wreath}). One of its main objects of study is the arithmetic of dynamical systems given by a pair $(\mathbb P^1(K),f)$, where $K$ is a global field and $f$ is a rational function on $\mathbb P^1$. Standard questions include the determination of the set of periodic and pre-periodic points, the determination of integral points in orbits, the structure of field extensions attached to the iterations of $f$, and many others (see for example \cite{silv} for a comprehensive introduction on the topic).

When looking at the iterates of a rational function, one can construct very naturally an infinite tree, carring a natural profinite topology, on which the absolute Galois group of $K$ acts continuously, giving rise to what is called an \emph{arboreal Galois representation} (cf.\ section~\ref{arbo} for details). These extremely interesting objects resemble in many aspects $p$-adic representations coming from geometry (see for example~\cite{jon1} for a survey), such as the Tate modules attached to elliptic curves. In the general case, not much is known about the behaviour of arboreal Galois representations, but a certain number of results are available when the rational function has degree two (see~\cite{bosjon},\cite{bosjon2},\cite{jon1} and \cite{jon3}). In particular, it seems that generically the image of these representations is ``large'', i.e.\ it has finite index in the group of automorphisms of the appropriate tree. This phenomenon recalls closely Serre's open image theorem for elliptic curves without complex multiplication \cite{serre}.

Focusing on rational functions $f$ which are actually polynomials with coefficients in the ring of integers of $K$ yields a greater number of arithmetic questions, such as the determination of the set of prime divisors in orbits (see \cite{jon4} and \cite{odoni2}) or, in the case where all iterates are irreducible, the determination of the set of primes $\p$ such that all iterations of $f$ are irreducible modulo $\p$. Following the terminology of the existing literature, we will call such primes \emph{stable}. In \cite{jon2}, the author focuses on the case of quadratic polynomials and conjectures that, under some hypoteses on the post-critical orbit of $f$, the set of stable primes is finite (see~\cite[Conjecture~6.2]{jon2}).

In this paper, we address the problem of finding the density of the set of stable primes under a precise condition of largeness of the attached arboreal Galois representation (cf.\ Theorem~\ref{main_thm}), but for a more general class of objects that we will now introduce. In particular, such setting can be specialized to that of a dynamical system given by a polynomial of any degree.

Let $K$ be a number field with number ring $\O_K$, and let $\{f_k\}_{k\in \N}\subseteq \O_K[x]$ be a sequence of polynomials. For every $n\in \N$, let $f^{(n)}\coloneqq f_1\circ\ldots\circ f_n$, and suppose $f^{(n)}$ is separable for all $n$. The study of arithmetic dynamical systems corresponds to the case where the sequence $\{f_k\}$ is constant. As in the classical case, one can attach an arboreal Galois representation to the sequence $\{f_k\}$, where the Galois group of $f^{(n)}$ acts on the set of vertices at level $n$. Now suppose that $f^{(n)}$ is irreducible for every $n$. We call a prime $\p\subseteq \O_K$ \emph{stable} for $\{f_k\}$ if $f^{(n)}$ is irreducible modulo $\p$ for every $n$. Our main theorem is then the following.
\begin{thm}\label{thm_intro}
 Suppose that the image of the arboreal Galois representation attached to $\{f_k\}$ is large enough. Then the set of stable primes for $\{f_k\}$ has density 0.
\end{thm}

In section~\ref{arbo} we introduce the arboreal Galois representation attached to $\{f_k\}$, we explain the condition of ``largeness'' mentioned in Theorem~\ref{thm_intro} and we show some of the consequences of the theorem. In section~\ref{autom} we describe the structure of the automorphism group of the infinite tree attached to $\{f_k\}$ and we compute the cardinality of the automorphism group of the tree truncated at level $n$. In section~\ref{main_proof} we prove Theorem~\ref{thm_intro}, explaining how it follows from Chebotarev density theorem together with the computations of section~\ref{autom}. Finally, in section~\ref{generic_case} we address the following question: suppose we fix a sequence of degrees $\{d_k\}_{k\in \N}$, and then we choose ``at random'' a sequence of polynomials $\{f_k\}$ such that $f_k$ has degree $d_k$ for every $k$. What are the odds that such sequence fulfils the hypotheses of Theorem~\ref{thm_intro}? After introducing an adequate concept of density on the set of all polynomial sequences of fixed degree, we prove that the density of the set of polynomial sequences that fulfil such hypotheses is 1. The result follows from theorems of Cohen \cite{cohen} and Odoni \cite{odoni}.

\section*{Acknowledgements}
We would like to thank Giacomo Micheli and Rafe Jones for their helpful comments. The author was supported by Swiss National Science Foundation grant number 168459.

\section{Arboreal Galois representations}\label{arbo}
Let $K$ be a field and let $\{f_k\}\coloneqq \{f_k\}_{k\in \N}\sub K[x]$ be a sequence of polynomials. For every $n\geq 1$, we set $d_n\coloneqq \deg f_n$ and we let $f^{(n)}$ be the composition $f_1\circ f_2\circ\ldots\circ f_n$. We assume that $f^{(n)}$ is separable for every $n$.

It is possible to attach to $\{f_k\}$ an infinite tree $\T$ in the following way: the root of the tree is labeled by $0$, and for every $n\geq 1$, the vertices at level $n$ are labeled by the roots of $f^{(n)}$ in $\overline{K}$. A vertex $\alpha$ at level $n$ descends from a vertex $\beta$ at level $n-1$ if and only if $f_n(\alpha)=\beta$. Note that thanks to the separability assumption, every vertex at level $n$ has exactly $d_{n+1}$ descendants at level $n+1$. Such a tree is called \emph{spherically homogeneous rooted tree} (see for example \cite{bar} and \cite{cec}), and it depends only on the sequence of the degrees $(d_1,d_2,\ldots, d_n,\ldots)$, which is called \emph{spherical index} of $\T$. When the spherical index is constant and equal to some $d\in \N$, $\T$ is called a \emph{complete rooted $d$-ary tree}.

For every $n\geq 1$, we let $\T_n$ be the tree truncated at level $n$. This consists of the finite tree formed by all vertices which have distance at most $n$ from the root.

\begin{figure}[h]
\centering

\begin{tikzpicture}
\tikzstyle{solid node}=[circle,draw,inner sep=1.5,fill=black]
\node[solid node]{}[sibling distance=40mm]
child{node[solid node]{}[sibling distance=13mm]
child{node[solid node]{}[sibling distance=7mm]
child{node[solid node]{}}
child{node[solid node]{}}}
child{node[solid node]{}[sibling distance=7mm]
child{node[solid node]{}}
child{node[solid node]{}}}
child{node[solid node]{}}[sibling distance=7mm]
child{node[solid node]{}}
child{node[solid node]{}}}
child{node[solid node]{} [sibling distance=13mm]
child{node[solid node]{}[sibling distance=7mm]
child{node[solid node]{}}
child{node[solid node]{}}}
child{node[solid node]{}[sibling distance=7mm]
child{node[solid node]{}}
child{node[solid node]{}}}
child{node[solid node]{}[sibling distance=7mm]
child{node[solid node]{}}
child{node[solid node]{}}}
};
\end{tikzpicture}\caption{A spherically homogeneous tree of spherical index $(2,3,2,\ldots)$ truncated at level 3.} \label{fig:M1}
\end{figure}
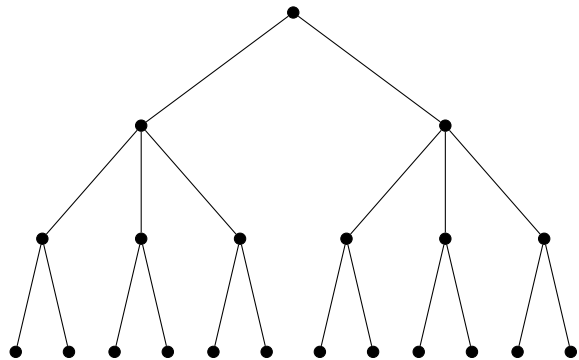
            
From now on, we let $G_n$ be the Galois group of $f^{(n)}$. The action of $G_n$ on the roots of $f^{(n)}$ extends naturally to an action on $\T_n$, and this yields an embedding of $G_n$ into the group of automorphisms of $\T_n$. Recall that an automorphism of a tree is a bijection $\sigma$ of the set of vertices such that a vertex $v$ is connected with a vertex $v'$ if and only if $\sigma(v)$ is connected with $\sigma(v')$. It follows immediately that a tree automorphism induces a permutation of the set of vertices at level $n$, for every $n$. Thus there are obvious projection maps $\pi_n\colon \aut(\T_n)\to \aut(\T_{n-1})$ and the automorphism group of $\T$ can be realized as
$$\aut(\T)\simeq \varprojlim_n\aut(\T_n).$$
On the other hand, it is immediate to check that for every $n$, the splitting field of $f^{(n)}$ is contained in the splitting field of $f^{(n+1)}$. Thus, there are surjections $G_{n+1}\to G_n$ and the profinite group $\disp G_{\{f_k\}}\coloneqq \varprojlim_nG_n$ acts on $\T$ as a subgroup of $\aut(\T)$, giving rise to a continuous embedding $G_{\{f_k\}}\to \aut(\T)$. This motivates the following definition, generalizing the one given in \cite{bosjon}.

\bdef
	An \emph{arboreal Galois representation} of a profinite group $G$ is a continuous homomorphism $G\to \aut(\T)$, where $\T$ is a spherically homogeneous rooted tree.
\ndef
In the particular case where $\{f_k\}$ is a constant sequence, one recovers \cite[Definition~1.1]{bosjon}.

From now on, $K$ will be a number field with ring of integers $\mathcal O_K$, and $f_k\in \O_K[x]$ will be monic for every $k\in \N$. We assume throughout the paper that $f^{(n)}$ is irreducible for every $n$. We also set $d^{(0)}=1$ and for every $n\geq 1$ we let $d^{(n)}\coloneqq \deg f^{(n)}$.
\bdef
	We say that a prime ideal $\p\subseteq \O_K$ is \emph{stable} for $\{f_k\}$ if $f^{(n)}$ is irreducible modulo $\p$ for every $n\geq 1$.

	The set of stable primes for $\{f_k\}$ will be denoted by $\St(\{f_k\})$.
\ndef
The main goal of this paper is to prove the following theorem.
\bthm\label{main_thm}
	Suppose that the index of $G_n$ in $\aut(\T_n)$ is $o(d^{(n)})$. Then $\St(\{f_k\})$ has density $0$.
\nthm

Recall that if $S$ is a set of prime ideals of $\O_K$, the \emph{natural density} of $S$ and the \emph{Dirichlet density} of $S$ are defined respectively as:
$$\lim_{x\to \infty}\frac{|\p\in S\colon N_{K/\Q}(\p)\leq x|}{|\p\sub \O_K\colon N_{K/\Q}(\p)\leq x|} \mbox{ and } \lim_{s\to 1^+}\frac{\sum_{\p\in S}N_{K/\Q}(\p)^{-s}}{\sum_{\p\sub\O_K}N_{K/\Q}(\p)^{-s}},$$
(provided that the limits exist). The word ``density'' in Theorem~\ref{main_thm} refers to either concept of density, since the result is true for both. The density of a set $S$ will be denoted by $\delta(S)$. The \emph{upper density} of $S$, which is defined using $\limsup$ in place of $\lim$ in the above formulas, will be denoted by $\overline{\delta}(S)$.

Theorem~\ref{main_thm} has the following immediate consequence.
\bcor\label{finite_index}
	Suppose that the sequence $\{d_k\}_{k\in \N}$ is not eventually 1 (i.e.\ that $d^{(n)}\to+\infty$) and that the arboreal Galois representation of $G_{\{f_k\}}$ has finite index image in $\aut(\T)$. Then the set of stable primes for $\{f_k\}$ has density $0$.
\ncor
\bpro
	Just note that since $G_n\leq \aut(\T_n)$ for every $n$, the group $G_{\{f_k\}}$ has finite index in $\aut(\T)$ if and only if the index of $G_n$ in $\aut(\T_n)$ is eventually constant, and therefore in particular it is $o(d^{(n)})$ as $d^{(n)}\to+\infty$.
\npro

\brmk
It is very easy to see that the conclusion of Theorem~\ref{main_thm} fails if we drop the assumption on the size of $G_n$. For example, if $\{f_k\}$ is the constant sequence with $f_k=x^2-2\in \Z[x]$, then the Galois group of $f^{(n)}$ is the cyclic group of order $2^n$ (see~\cite{bosjon}). Now \cite[Theorem~2.2]{jon2} shows that, for every fixed $n$, $f^{(n)}$ is irreducible modulo $p$ if and only if $p\equiv 3,5\bmod 8$. Thus in this case, the set of stable primes for $\{f_k\}$ has density $1/2$.
\nrmk
Clearly, one can use Theorem~\ref{main_thm} in its contrapositive form to prove that the index of $G_{\{f_k\}}$ in $\aut(\T)$ is infinite. An explicit example can be constructed as follows. Let $p$ be any prime, and let $f_k\coloneqq (x-p^{2k+1})^2+p^{2k-1}$ for every $k\geq 1$. Then \cite[Proposition~3.3]{fermic1} shows that if $q$ is a prime with $\displaystyle\left(\frac{p}{q}\right)=\left(\frac{-p}{q}\right)=-1$, then the composition $f^{(n)}$ is irreducible modulo $q$ for every $n\in \N$. Quadratic reciprocity easily implies the existence of a set of positive density of such $q$'s, showing that the index of $G_{\{f_k\}}$ in $\aut(\T)$ is infinite.

The same argument furnishes a different proof of some cases of \cite[Theorem 3.1]{jon1}: if $f\in \O_K[x]$ is a polynomial of degree 2 such that all its iterates are irreducible and the span of its post-critical orbit in $K^*/{K^*}^2$ is finite and does not contain the origin, then the set of stable primes for $f$ has positive density (see \cite[Theorem 6.1]{jon2}). Thus, by Theorem~\ref{main_thm} the attached Galois representation has infinite index in $\aut(\T)$. This applies for example to the polynomial $(x-t)^2 + t + 1\in \Z[x]$, where $t\in\Z$ is such that $\pm t$ and $\pm (t+1)$ are all non-squares.

When $\{f_k\}$ is a constant sequence of spherical index 2, we fall back in the setting studied for example in \cite{bosjon},\cite{jon2},\cite{jon1} or \cite{jon3}. Let us briefly recall such setting. Let $\phi\in K(x)$ be a rational function of degree 2. A \emph{critical point} $\gamma$ of $\phi$ is a point $\gamma\in\mathbb P^1$ such that $\phi'(\gamma)=0$. Recall that $\phi$ is said to be \emph{post-critically finite} if the orbit of every critical point under $\phi$ is finite. Generalizing in the obvious way the construction we discussed above, one attaches a complete rooted $2$-ary tree $\T$ to $\phi$, and there is a continuous action of the absolute Galois group of $K$ on it, giving rise to an arboreal Galois representation. Let $G_{\phi}$ be the image of such representation.
\begin{conj}[{\cite[Conjecture~3.11]{jon1}}]\label{conj}
 The index of $G_{\phi}$ in $\aut(\T)$ is finite if and only if one of the following holds:
 \begin{enumerate}
  \item The map $\phi$ is post-critically finite.
  \item The two critical points $\gamma_1$ and $\gamma_2$ of $\phi$ have a relation of the form $\phi^{(r+1)}(\gamma_1)=\phi^{(r+1)}(\gamma_2)$ for some $r\geq 1$.
  \item $0$ is periodic under $\phi$.
  \item There is a non-trivial M\"obius transformation $m$ that fixes $0$ and such that $\phi\circ m=m\circ \phi$.
 \end{enumerate}

\end{conj}
Let now $f\in \O_K[x]$ be a monic polynomial of degree 2, let $f_k=f$ for every $k$ and assume $f^{(n)}$ is irreducible for every $n$. 
\bcor
  Let $f$ be as above and assume Conjecture~\ref{conj}. If $f$ is not post-critically finite, then the set of stable primes for $f$ has density $0$. 
\ncor
\begin{proof}
 By Corollary~\ref{finite_index}, it is enough to check that conditions (2),(3) and (4) are never satisfied by such a polynomial. Condition (2) clearly does not hold because $f$ has two distinct critical points, one of which is $\infty$, and so it is fixed by $f$, and the other one is never mapped to $\infty$ by any iterate of $f$. Condition (3) cannot hold because $f^{(n)}$ is irreducible for every $n$ by assumption. A direct computation shows that if (4) holds for an irreducible polynomial, then this polynomial must be $(x-1)^2+1$, which is post-critically finite (and conversely, such polynomial commutes with $x/(x-1)$).
\end{proof}

\section{The automorphism group of \texorpdfstring{$\T_n$}{}}\label{autom}

In order to describe the group of automorphisms of $\T_n$, we first recall the construction of the wreath product of groups. Let $G,H$ be two groups and $R$ be a set on which $G$ acts by permutations from the left. If $g\in G$ and $r\in R$, we denote by $g\cdot r$ the action of $g$ on $r$. Let $H^R\coloneqq \prod_{r\in R}H_r$, where each $H_r$ is an isomorphic copy of $H$. The action of $G$ extends naturally to $H^R$ via $g\cdot (h_r)_{r\in R}=(h_{g^{-1}\cdot r})_{r\in R}$. This defines a homomorphism 
$$\Phi \colon G\to \text{Aut}(H^R)$$
$$g\mapsto (\varphi_g\colon (h_r)_r\mapsto g\cdot (h_r)_r).$$
The \emph{wreath product} of $G$ by $H$ is defined by:
$$G\wr_R H\coloneqq G\ltimes_{\Phi} H^R.$$
Now suppose that $G$ is a subgroup of the symmetric group on $d$ symbols $S_d$ and $H$ is a subgroup of $S_e$. Let $R\coloneqq\{1,\ldots,d\}$ and $T=\{1,\ldots,e\}$. Then $G\wr_R H$ acts from the left on $R\times T$ via the following rule:
$$(g,(h_r)_r)\cdot (r_0,t_0)=(g\cdot r_0,h_{g\cdot r_0}\cdot t_0).$$

Let now $\T$ be the spherically homogeneous tree of spherical index $\{d_k\}_{k\in \N}$.
\bthm[{\cite[Theorem~2.1.15]{cec}}]\label{wreath}
	The automorphism group of $\T_n$ is isomorphic to the wreath product:
	$$S_{d_1}\wr S_{d_2}\wr\ldots\wr S_{d_n}$$
\nthm
The wreath product of groups is associative and therefore Theorem~\ref{wreath} implies that we can think of $\aut(\T_n)$ as $\aut(\T_{n-1})\wr S_{d_n}$; we will make use of this fact in what follows.
From now on, we denote by $W_n$ the automorphism group of $\T_n$.
\bcor\label{wreath_card}
	The cardinality of $W_n$ is $\disp\prod_{i=1}^n{(d_i!)}^{d^{(i-1)}}$.
\ncor
\bpro
	If $G,H$ are finite groups with $G$ acting on a finite set $R$, it is clear from the definition that $|G\wr_R H|=|G|\cdot |H|^{|R|}$. The claim follows by an easy induction using Theorem~\ref{wreath} and the fact that $S_{d_1}\wr S_{d_2}\wr\ldots\wr S_{d_n}$ acts on $d^{(n)}$ symbols.
\npro
\brmk\label{tree_action}
	It is useful to understand how elements of $W_n$ act on the vertices of $\T_n$. Notice that $W_n$ is a subgroup of $S_{d^{(n)}}$ by construction, and the set $V_n$ of the vertices at level $n$ has cardinality $d^{(n)}$. In fact, since automorphisms of $\T_n$ preserve connected vertices, in order to describe the action of $W_n$ on the vertices of $\T_n$ it is enough to specify the action of $W_n$ on $V_n$. For every $i\in \{1,\ldots,n\}$, let $R_i$ be the set $\{1,\ldots,d_i\}$. Each element $v\in V_n$ can be uniquely identified by a sequence $(t_1,\ldots,t_n)$ where $t_i\in R_i$ for every $i$. This identification comes from labeling the descendants of a vertex at level $i$ with the elements of $R_{i+1}$, so that the sequence $(t_1,\ldots,t_n)$ describes the unique path from the root of the tree to the corresponding vertex in $V_n$. Now, elements of $W_n$ are of the form $(g,(h_r)_{r\in R^{(n-1)}})$, where $g\in W_{n-1}$, $h_r\in S_{d_n}$ for every $r$ and $R^{(n-1)}=R_1\times\ldots\times R_{n-1}$. Thus, $g$ acts inductively on the sequence $(t_1,\ldots,t_{n-1})$, yielding a new sequence $(t_1',\ldots,t_{n-1}')$, and we have that
	$$(g,(h_r)_{r\in R^{(n-1)}})\cdot (t_1,\ldots,t_n)=(t_1',\ldots,t_{n-1}',h_{(t_1',\ldots,t_{n-1}')}\cdot t_n).$$
\nrmk

\section{Proof of the main theorem}\label{main_proof}
Our goal is to show that $\delta(\St(\{f_k\}))=0$.

Let $\St(f^{(n)})$ be the set of primes $\p$ such that $f^{(n)}$ is irreducible modulo $\p$. Notice that $\St(f^{(i)})\sub \St(f^{(j)})$ whenever $i\geq j$. It is immediate to see that
$$\St(\{f_k\})=\bigcap_{n\in \N}\St(f^{(n)}).$$
Since $\St(\{f_k\})\sub \St(f^{(n)})$ for every $n$, then $\overline{\delta}(\St(\{f_k\}))\leq \overline{\delta}(\St(f^{(n)}))$ for every $n$. Therefore, in order to prove Theorem~\ref{main_thm} it is enough to show that $\delta(\St(f^{(n)}))$ exists for every $n$ and converges to $0$ as $n\to \infty$. In fact, if this happens then $\overline{\delta}(\St(f^{(n)}))$ converges to $0$ as well, forcing $\overline{\delta}(\St(f))=0$ and finally $\delta(\St(f))=0$, by the obvious fact that $\delta(\St(f))\leq\overline{\delta}(\St(f))$. 

Thus, we reduced the proof of Theorem~\ref{main_thm} to proving the following claim:
\begin{center}
	($\spadesuit$) The density of $\St(f^{(n)})$ exists and converges to $0$ as $n\to\infty$.
\end{center}

Let us now recall the following fundamental theorem, which is a weaker version of Chebotarev's density theorem.
\bthm[Frobenius density theorem]\label{Frobenius}
Let $g(x)\in \O_K[x]$ be monic and irreducible of degree $d$. Let $G\sub S_d$ be the Galois group of $g$. Let $a_1\leq a_2\leq\ldots\leq a_t$ be natural numbers such that $a_1+\ldots+ a_t=d$. Let $\Gamma\sub G$ be the set of elements whose decomposition in disjoint cycles has the form $c_1\cdot c_2\cdot\ldots\cdot c_t$, where $c_i$ is a cycle of length $a_i$. Then the set
$$\{\mbox{primes $\p\sub \O_K$ s.t.\ $g(x)$ has decomposition type $(a_1,\ldots,a_t)$ modulo $\p$}\}$$
has density $\disp\frac{|\Gamma|}{|G|}$.
\nthm
Coming back to our setting, the following corollary is immediate. Recall that $G_n$ is the Galois group of $f^{(n)}$, and is therefore a subgroup of $S_{d^{(n)}}$.
\bcor\label{frob_cor}
Let $\Gamma_n\sub G_n$ be the subset of cycles of length $d^{(n)}$. Then $\displaystyle\delta(\St(f^{(n)}))=\frac{|\Gamma_n|}{|G_n|}$.
\ncor
\bpro
The claim follows from Theorem~\ref{Frobenius} together with the fact that $f^{(n)}$ has degree $d^{(n)}$, and therefore it is irreducible modulo a prime ideal $\p$ if and only if its decomposition type modulo $\p$ is $(d^{(n)})$.
\npro
The key lemma which allows us to prove claim ($\spadesuit$) is the following. Recall that $W_n=\text{Aut}(\T_n)$.
\blem\label{main_lem}
 Let $C_n\sub W_n$ be the set of cycles of length $d^{(n)}$. Then $\disp \frac{|C_n|}{|W_n|}=\frac{1}{d^{(n)}}$.
\nlem
Let us first show that Lemma~\ref{main_lem} implies claim ($\spadesuit$). Since $G_n\leq W_n$, the set $\Gamma_n$ of Corollary~\ref{frob_cor} is a subset of $C_n$. Thus, by Corollary~\ref{frob_cor} we get that
$$\delta(\St(f^{(n)}))=\frac{|\Gamma_n|}{|G_n|}\leq \frac{|C_n|}{|G_n|}=\frac{|C_n|}{|W_n|}\cdot [W_n : G_n].$$
Now Lemma~\ref{main_lem}, together with the fact that, by hypothesis, $[W_n : G_n]=o(d^{(n)})$, implies that $\delta(\St(f^{(n)}))\to 0$.
\bpro[Proof of Lemma~\ref{main_lem}]
By Corollary~\ref{wreath_card}, the statement of the lemma is equivalent to proving that:
$$|C_n|=\prod_{i=1}^n(d_i-1)! {(d_i!)}^{d^{(i-1)}-1}.$$
We will prove this by induction on $n$. For $n=1$, the claim is true because $W_1$ is the symmetric group on $d_1$ symbols, and thus it contains exactly $(d_1-1)!$ cycles of length $d_1$. In order to prove the claim for $C_n$, we need a characterization of cycles of length $d^{(n+1)}$ inside $W_{n+1}$. For every $i\in \{1,\ldots,n\}$, let $R_i\coloneqq\{1,\ldots,d_i\}$ and let $R^{(n)}=R_1\times\ldots\times R_n$. Recall that, by Theorem~\ref{wreath}, $W_{n+1}=W_n\wr_{R^{(n)}}S_{d_{n+1}}$ for $n\geq 1$, and that for every $n$, $W_n$ is naturally a subgroup of $S_{d^{(n)}}$ acting on the set of vertices of $\T$ at level $n$ as explained in Remark~\ref{tree_action}. Let $\overline{g}=(g,(h_r)_{r\in R^{(n)}})\in W_{n+1}$. We claim that $\overline{g}\in C_{n+1}$ if and only if the following two conditions hold:
\begin{enumerate}
 \item $g\in C_n$;
 \item for every $r\in R^{(n)}$, the element $\disp\prod_{i=1}^{d^{(n)}}h_{g^{-i}\cdot r}\in S_{d_{n+1}}$ is a cycle of length $d_{n+1}$.
\end{enumerate}
To prove it, let us first assume that $\overline{g}\in C_{n+1}$ and suppose that there is a cycle of length $a$ in the decomposition of $g$ into disjoint cycles. Then there is a vertex $v$ of $\T$ at level $n$ such that $g^a\cdot v=v$, and therefore $\overline{g}^a$ permutes the $d_{n+1}$ vertices at level $n+1$ that descend from $v$. If such a permutation contains a cycle of length $b$, it follows that there exists a vertex $w$, descending from $v$, such that $\overline{g}^{ab}\cdot w=w$. Since $a\leq d^{(n)}$ and $b\leq d_{n+1}$, equalities must hold because $\overline{g}$ is a cycle of length $d^{(n+1)}$ and therefore no vertex at level $n+1$ can be mapped to itself with less than $d^{(n+1)}$ iterations of $\overline{g}$. This argument shows that $g\in C_n$ and that since $\disp \overline{g}^{d^{(n)}}=\left(\id,\left(\prod_{i=1}^{d^{(n)}}h_{g^{-i}\cdot r}\right)_{r\in R^{(n)}}\right)$, then $\disp \prod_{i=1}^{d^{(n)}}h_{g^{-i}\cdot r}$ is a cycle of length $d_{n+1}$ for every $r\in R^{(n)}$.

Conversely, let $\overline{g}=(g,(h_r)_{r\in R^{(n)}})\in W_{n+1}$ have properties (1) and (2). Suppose that $\overline{g}$ contains a cycle of length $a$. Then there is a vertex $v$ of $\T$ at level $n+1$ such that $\overline{g}^a\cdot v=v$, which implies in particular that $\overline{g}^a\cdot v$ has the same parent of $v$. Since $g$ acts on the set of vertices at level $n$ and is a cycle of length $d^{(n)}$, this proves that $a=d^{(n)}b$, for some $b\leq d_{n+1}$. Since $\disp \overline{g}^{d^{(n)}}=\left(\id,\left(\prod_{i=1}^{d^{(n)}}h_{g^{-i}\cdot r}\right)_{r\in R^{(n)}}\right)$, it follows that there exists some $r_0\in R^{(n)}$ such that $\disp \prod_{i=1}^{d^{(n)}}h_{g^{-i}\cdot r_0}$ permutes the vertices with the same parent of $v$. This permutation is a cycle of length $d_{n+1}$ by (2) and this proves, together with the fact that $\overline{g}^a\cdot v=v$, that $b=d_{n+1}$, and finally that $\overline{g}\in C_{n+1}$.

We are now ready to enumerate the elements in $C_{n+1}$. Let $g$ be an elment of $C_n$ and fix $r_0\in R^{(n)}$. For every $i\in \{1,\ldots,d^{(n)}\}$ choose $h_{g^{-i}\cdot r_0}\in S_{d_{n+1}}$ such that the element
$$h\coloneqq h_{g^{-1}\cdot r_0}\cdot h_{g^{-2}\cdot r_0}\cdot \ldots\cdot h_{g^{-d^{(n)}}\cdot r_0}$$
is a cycle of length $d_{n+1}$ (notice that since $g$ is a cycle of maximal length, the set $\{g^{-1}\cdot r_0,\ldots,g^{-d^{(n)}}\cdot r_0\}$ coincides with $R^{(n)}$). We claim that $\overline{g}\coloneqq (g,(h_r)_{r\in R^{(n)}})$ is a cycle of length $d^{(n+1)}$. By the characterization that we proved above, this is equivalent to prove that
$\disp h_r\coloneqq \prod_{i=1}^{d^{(n)}}h_{g^{-i}\cdot r}$ is a cycle of length $d_{n+1}$ for every $r\in R^{(n)}$. Let $j\in \{1,\ldots,d^{(n)}\}$ be the unique element such that $g^{-j}\cdot r=r_0$. Then $h=h_{g^{-j-1}\cdot r}\cdot h_{g^{-j-2}\cdot r}\cdot\ldots\cdot h_{g^{-j-d^{(n)}}\cdot r}$, and setting
$$k\coloneqq h_{g^{-1}\cdot r}\cdot h_{g^{-2}\cdot r}\cdot\ldots\cdot h_{g^{-j}\cdot r},$$
we have the equality
$$h_r=khk^{-1}.$$
This proves that $h$ and $h_r$ are conjugate in $S_{d_{n+1}}$, and therefore also $h_r$ is a cycle of length $d_{n+1}$.

In other words, we have proved that for every $g\in C_n$, in order to construct an element $(g,(h_r)_{r\in R^{(n)}})\in W_{n+1}$ lying in $C_{n+1}$ it is necessary and sufficient to fix $r_0\in R^{(n)}$ and to find, for every $i\in\{1,\ldots,d^{(n)}\}$, an element $h_{g^{-i}\cdot r_0}\in S_{d_{n+1}}$ such that $\disp\prod_{i=1}^{d^{(n)}}h_{g^{-i}\cdot r_0}$ is a cycle of length $d_{n+1}$. For every $g\in C_n$, we have complete freedom in choosing $h_{g^{-1}\cdot r_0},h_{g^{-2}\cdot r_0},\ldots,h_{g^{-d^{(n)}+1}\cdot r_0}$, which means that we have ${(d_{n+1}!)}^{d^{(n)}-1}$ choices; we must then have
$$h_{g^{-d^{(n)}}\cdot r_0}=\left(\prod_{i=1}^{d^{(n)}-1}h_{g^{-i}\cdot r_0}\right)^{-1}\cdot c,$$
where $c$ is a cycle of length $d_{n+1}$. This means that we are left with $(d_{n+1}-1)!$ choices for $c$, because this is the number of cycles of length $d_{n+1}$. All in all, we have $(d_{n+1}-1)! {(d_{n+1}!)}^{d^{(n)}-1}$ choices for every element $g\in C_n$. Since by the induction hypothesis we have that $|C_{n}|=\prod_{i=1}^n(d_i-1)!{(d_i!)}^{d^{(i-1)}-1}$, we easily get that
$$|C_{n+1}|=(d_{n+1}-1)!{(d_{n+1}!)}^{d^{(n)}-1}\cdot \prod_{i=1}^n(d_i-1)!{(d_i!)}^{d^{(i-1)}-1}=\prod_{i=1}^{n+1}(d_i-1)!\cdot {(d_i!)}^{d^{(i-1)}-1},$$
as desired.
\npro

\section{The generic case}\label{generic_case}
It is a very hard problem, in general, to compute explicitly the Galois groups $G_n=\gal(f^{(n)})$ for a given sequence $\{f_k\}\subseteq \O_K[x]$, even when such sequence is constant (see \cite{bened}, \cite{bosjon2} or \cite{stoll} for examples in degree $2$ and $3$). It is natural to ask what is the generic behaviour of a sequence of fixed spherical index. In this section we will prove that, in an adequate sense, the set of sequences $\{f_k\}$ whose associated arboreal Galois representation is surjective has density $1$. This shows in particular that the set of sequences that fulful the hypotheses of Theorem~\ref{main_thm} has density $1$.

Let us first recall the notion of natural density for subsets of $\O_K^n$ (cf.\ \cite{fermic2}). Let $m\coloneqq [K\colon \Q]$, fix a $\Z$-basis $\mathcal B\coloneqq\{\omega_1,\ldots,\omega_m\}$ for $\O_K$ and define
$$\O_K[N,\mathcal B]\coloneqq \left\{\sum_{i=1}^ma_i\omega_i\in\O_K\colon |a_i|\leq N \,\,\forall i\in\{1,\ldots,m\}\right\}.$$
The density of a subset $A\subseteq \O_K^n$ (with respect to $\mathcal B$) is defined as
$$\D(A)\coloneqq\lim_{N\to \infty}\frac{|A\cap \O_K[N,\mathcal B]^n|}{|\O_K[N,\mathcal B]^n|},$$
provided that the limit exists. As $\O_K^n$ is a countably infinite set, there is no uniform probability distribution on it. The above notion of density is to be thought as the limit, as $N\to\infty$, of the probability that a point chosen uniformly at random inside the $mn$-dimensional hypercube of side $N$ and centered in the origin belongs to $A$, after chosing an identification of $\O_K^n$ with $\Z^{mn}$.

From now on, we will fix a spherical index $\{d_k\}_{k\in \N}$. For every $n\in \N$, let $\X_n\coloneqq \prod_{i=1}^n\O_K^{d_i}$ and let $\mathcal X\coloneqq \prod_{i=1}^{\infty}\O_K^{d_i}$. The set $\X_n$ can be naturally identified with the set of $n$-tuples of monic polynomials $\{f_1,\ldots,f_n\}$ such that each $f_i$ has degree $d_i$, simply by mapping each $f_i$ to the $d_i$-tuple of its coefficients in $\O_K^{d_i}\subseteq \X_n$. Analogously, the set $\X$ can be identified with the set of monic polynomial sequences $\{f_k\}\subseteq \O_K[x]$ of spherical index $\{d_k\}$. We will assume these identifications implicitly in what follows. Identifying $\X_n$ with $\O_K^{\sum_{i=1}^n d_i}$ in the obvious way, we get a well-defined notion of density on $\X_n$. Let $\pi_n\colon \X \to \X_n$ be the natural projection map.
\bdef
The density of a subset $A\subseteq \X$ (with respect to $\mathcal B$) is defined as
$$\lim_{n\to\infty}\D(\pi_n(A)),$$
provided that the limit exists.
\ndef

Again, this does not define a probability distribution on $\X$. However, one can think of choosing a sequence $\{f_k\}$ as an analogue of a discrete stochastic process; our definition of density on $\X$ serves then as an analogous of the concept of joint distribution of the process.

Recall that we denote by $W_n$ the full automorphism group of the spherically homogeneous tree of spherical index $\{d_k\}$ truncated at level $n$, and that for a polynomial sequence $\{f_k\}$, we denote by $G_n$ the Galois group of $f^{(n)}=f_1\circ\ldots\circ f_n$. The goal of this section is to prove the following theorem.
\begin{thm}\label{density}
 Let $A\subseteq \X$ be the set of all polynomial sequences $\{f_k\}$ such that $G_n \simeq W_n$ for every $n\in \N$. Then $\D(\pi_n(A))=1$ for every $n$, and therefore $\D(A)=1$.
\end{thm}
To prove the theorem, we need to recall the following results.
\begin{thm}[{\cite[Corollary~8.4]{odoni}}]\label{odo}
 Let $F$ be a field of characteristic 0, and let $f\in F[x]$ be monic and squarefree with Galois group $G$ over $F$. For every $\ell\geq 2$, let $t_1,\ldots,t_{\ell}$ be indeterminates over $F$ and let $g(x,t_1,\ldots,t_{\ell})\coloneqq x^{\ell}+t_1x^{\ell-1}+\ldots+t_{\ell}\in F[x,t_1,\ldots,t_{\ell}]$. Then the Galois group of $f\circ g$ over $F(t_1,\ldots,t_{\ell})$ is isomorphic to the wreath product $G\wr S_{\ell}$.
\end{thm}
The following theorem is stated in greater generality in \cite{cohen}. We report it here in a simpler version which sufficies for our purposes.
\begin{thm}\label{coh}
 Let $K$ be a number field, let $x,t_1,\ldots,t_{\ell}$ be indeterminates over $K$ and let $f(x,t_1,\ldots,t_{\ell})\in \O_K[x,t_1,\ldots,t_{\ell}]$ have Galois group $G$ over $K(t_1,\ldots,t_{\ell})$. Then there exist constants $c_1,c_2$, depending on $f,K$ and $\mathcal B$, such that for all $N>c_1$, the number of $\ell$-tuples $(\alpha_1,\ldots,\alpha_\ell)\in \O_K[N,\mathcal B]^{\ell}$ such that the Galois group of $f(x,\alpha_1,\ldots,\alpha_{\ell})$ over $K$ is not isomorphic to $G$ does not exceed $c_2N^{m(\ell-1/2)}\log N$. 
\end{thm}
\begin{proof}
 See \cite[Theorem~2.1]{cohen}. We remark that the set denoted by $\O_K[N,\mathcal B]$ by us, coincides with the set denoted by $\Z_K(N^m)$ in~\cite{cohen}.
\end{proof}
\begin{proof}[Proof of Theorem~\ref{density}]
 We first notice that for every $n\in \N$, we have that
 $$\pi_n(A)=\{(f_1,\ldots,f_n)\in \X_n\colon \gal(f^{(i)})\simeq W_i \mbox{ for every }i\in \{1,\ldots,n\}\}.$$
 In fact, by definition the set $\pi_n(A)$ coincides with the set of $n$-tuples of polynomials $(f_1,\ldots,f_n)$ that satisfy the following two conditions:
 \begin{enumerate}[i)]
  \item $G_i\simeq W_i$ for every $i\in \{1,\ldots,n\}$;
  \item there exists a sequence of polynomials $\{f_{n+k}\}_{k\in \N}$ of spherical index $\{d_{n+k}\}_{k\in \N}$ such that $G_i\simeq W_i$ for every $i>n$.
 \end{enumerate}
 Theorems \ref{odo} and \ref{coh} show that for every $n$-tuple of polynomials satisfying i) it is possible to construct a sequence (and in fact infinitely many) $\{f_{n+k}\}_{k\in \N}$ satisfying ii): let $(f_1,\ldots,f_s)\in \X_s$ for some $s\geq n$ be such that $G_i\simeq W_i$ for every $i\in \{1,\ldots,s\}$ and let $g\coloneqq x^{d_{s+1}}+t_1x^{d_{s+1}-1}+\ldots+t_{d_{s+1}}\in K[x,t_1,\ldots,t_{d_{s+1}}]$. Then by Theorem~\ref{odo}, the polynomial $f^{(s)}\circ g$ has Galois group $W_{s+1}$ over $K(t_1,\ldots,t_{d_{s+1}})$ and by Theorem~\ref{coh} there exist infinitely many specializations $(\alpha_1,\ldots,\alpha_{d_{s+1}})\in \O_K^{d_{s+1}}$ such that $(f^{(s)}\circ g)(x,\alpha_1,\ldots,\alpha_{d_{s+1}})$ has Galois group $W_{s+1}$. Thus, it is enough to set $f_{s+1}\coloneqq g(x,\alpha_1,\ldots,\alpha_{d_{s+1}})$ and to apply the same argument inductively to obtain a sequence $\{f_{n+k}\}_{k\in \N}$ satisfying ii).
 
 To compute the density of $\pi_n(A)$, for every $i\in \{1,\ldots,n\}$ set
 $$g_i(x,t_1^{(i)},\ldots,t_{d_i}^{(i)})\coloneqq x^{d_i}+t_1^{(i)}x^{d_i-1}+\ldots+t_{d_i}^{(i)}\in K(t_1^{(i)},\ldots,t_{d_i}^{(i)})[x].$$
 Here $x$ and the $t_j^{(i)}$'s are algebraically independent indeterminates over $K$. It is a well-known fact (see for example~\cite[Corollary~7.3]{odoni}) that the Galois group of $g_1$ over $K(t_1^{(1)},\ldots,t_{d_1}^{(1)})$ is isomorphic to $S_{d_1}$. Now applying Theorem~\ref{odo} with $F=K(t_1^{(1)},\ldots,t_{d_1}^{(1)})$ it follows that the Galois group of $g_1\circ g_2$ over $K(t_1^{(1)},\ldots,t_{d_1}^{(1)},t_1^{(2)},\ldots,t_{d_2}^{(2)})$ is $S_{d_1}\wr S_{d_2}$. Repeating inductively the same argument shows that for every $i\in \{1,\ldots,n\}$ the Galois group of $g^{(i)}\coloneqq g_1\circ g_2\circ\ldots\circ g_i$ over $K(t_1^{(1)},\ldots,t_{d_i}^{(i)})$ is $W_i$. Letting $D_i\coloneqq \sum_{j=1}^id_j$ for every $i\in \{1,\ldots,n\}$, we therefore have that the set $\pi_n(A)\cap \O_K[N,\mathcal B]^{D_n}$ coincides with the set
 $$\{(\alpha_j)_{j=1}^{D_n}\in \O_K[N,\mathcal B]^{D_n}\colon \gal(g^{(i)}(x,\alpha_1,\ldots,\alpha_{D_i}))\simeq W_i\,\,\forall\,i\in\{1,\ldots,n\}\}.$$
 By Theorem~\ref{coh}, we can find constants $c_1,c_2$, depending on all the $g_i$'s, on $\mathcal B$ and on $K$, such that for all $N>c_1$ and for all $i\in\{1,\ldots,n\}$, the number of $(\alpha_1,\ldots,\alpha_{D_i})\in \O_K[N,\mathcal B]^{D_i}$ such that the Galois group of $g^{(i)}(\alpha_1,\ldots,\alpha_{D_i})$ is not $W_i$ does not exceed $c_2N^{m(D_i-1/2)}\log N$. Letting
 $$B_i\coloneqq \{(\alpha_1,\ldots,\alpha_{D_n})\in \O_K[N,\mathcal B]^{D_n}\colon  \gal(g^{(i)}(x,\alpha_1,\ldots,\alpha_{D_i}))\not\simeq W_i\}$$
 for every $i\in \{1,\ldots,n\}$, it follows that the cardinality of $B_i$ is bounded by $c_2N^{m(D_n-1/2)}\log N$. Since $\O_K[N,\mathcal B]^{D_n}\setminus\pi_n(A)\subseteq \bigcup_{i=1}^n B_i$, we have that
 $$|\pi_n(A)\cap \O_K[N,\mathcal B]^{D_n}|\geq |\O_K[N,\mathcal B]^{D_n}|-\sum_{i=1}^n|B_i|\geq |\O_K[N,\mathcal B]^{D_n}|-nc_2N^{m(D_n-1/2)}\log N,$$
 and the claim follows simply by the fact that $|\O_K[N,\mathcal B]^{D_n}|=(2N+1)^{mD_n}$.
 
\end{proof}



\bibliographystyle{plain}
\bibliography{bibliografia}

\begin{thebibliography}{10}

\bibitem{bar}
Laurent Bartholdi, Rostislav Grigorchuk, and Volodymyr Nekrashevych.
\newblock From fractal groups to fractal sets.
\newblock In {\em Fractals in {G}raz 2001}, Trends Math., pages 25--118.
  Birkh\"auser, Basel, 2003.

\bibitem{bened}
Robert~L. Benedetto, Xander Faber, Benjamin Hutz, Jamie Juul, and Yu~Yasufuku.
\newblock A large arboreal galois representation for a cubic postcritically
  finite polynomial.
\newblock \url{https://arxiv.org/abs/1612.03358}, 2016.

\bibitem{bosjon}
Nigel Boston and Rafe Jones.
\newblock Arboreal {G}alois representations.
\newblock {\em Geom. Dedicata}, 124:27--35, 2007.

\bibitem{bosjon2}
Nigel Boston and Rafe Jones.
\newblock The image of an arboreal {G}alois representation.
\newblock {\em Pure Appl. Math. Q.}, 5(1):213--225, 2009.

\bibitem{cec}
Tullio Ceccherini-Silberstein, Fabio Scarabotti, and Filippo Tolli.
\newblock {\em Representation theory and harmonic analysis of wreath products
  of finite groups}, volume 410 of {\em London Mathematical Society Lecture
  Note Series}.
\newblock Cambridge University Press, Cambridge, 2014.

\bibitem{cohen}
S.~D. Cohen.
\newblock The distribution of {G}alois groups and {H}ilbert's irreducibility
  theorem.
\newblock {\em Proc. London Math. Soc. (3)}, 43(2):227--250, 1981.

\bibitem{fermic2}
Andrea Ferraguti and Giacomo Micheli.
\newblock On the {M}ertens-{C}es\`aro theorem for number fields.
\newblock {\em Bull. Aust. Math. Soc.}, 93(2):199--210, 2016.

\bibitem{fermic1}
Andrea Ferraguti, Giacomo Micheli, and Reto Schnyder.
\newblock Irreducible compositions of degree two polynomials over finite fields
  have regular structure.
\newblock \url{https://arxiv.org/abs/1701.06040}, 2017.

\bibitem{fermic3}
Andrea Ferraguti, Giacomo Micheli, and Reto Schnyder.
\newblock On sets of irreducible polynomials closed by composition.
\newblock In {\em Arithmetic of finite fields}, volume 10064 of {\em Lecture
  Notes in Comput. Sci.}, pages 77--83. Springer, Cham, 2017.

\bibitem{micheli}
D.~R. Heath-Brown and Giacomo Micheli.
\newblock Irreducible polynomials over finite fields produced by composition of
  quadratics.
\newblock \url{https://arxiv.org/abs/1701.05031}, 2017.

\bibitem{jon4}
Rafe Jones.
\newblock The density of prime divisors in the arithmetic dynamics of quadratic
  polynomials.
\newblock {\em J. Lond. Math. Soc. (2)}, 78(2):523--544, 2008.

\bibitem{jon2}
Rafe Jones.
\newblock An iterative construction of irreducible polynomials reducible modulo
  every prime.
\newblock {\em J. Algebra}, 369:114--128, 2012.

\bibitem{jon1}
Rafe Jones.
\newblock Galois representations from pre-image trees: an arboreal survey.
\newblock In {\em Actes de la {C}onf\'erence ``{T}h\'eorie des {N}ombres et
  {A}pplications''}, Publ. Math. Besan\c con Alg\`ebre Th\'eorie Nr., pages
  107--136. Presses Univ. Franche-Comt\'e, Besan\c con, 2013.

\bibitem{jon3}
Rafe Jones and Michelle Manes.
\newblock Galois theory of quadratic rational functions.
\newblock {\em Comment. Math. Helv.}, 89(1):173--213, 2014.

\bibitem{wreath}
Jamie Juul, P\"ar Kurlberg, Kalyani Madhu, and Tom~J. Tucker.
\newblock Wreath products and proportions of periodic points.
\newblock {\em Int. Math. Res. Not. IMRN}, (13):3944--3969, 2016.

\bibitem{odoni}
R.~W.~K. Odoni.
\newblock The {G}alois theory of iterates and composites of polynomials.
\newblock {\em Proc. London Math. Soc. (3)}, 51(3):385--414, 1985.

\bibitem{odoni2}
R.~W.~K. Odoni.
\newblock On the prime divisors of the sequence {$w_{n+1}=1+w_1\cdots w_n$}.
\newblock {\em J. London Math. Soc. (2)}, 32(1):1--11, 1985.

\bibitem{serre}
Jean-Pierre Serre.
\newblock Propri\'et\'es galoisiennes des points d'ordre fini des courbes
  elliptiques.
\newblock {\em Invent. Math.}, 15(4):259--331, 1972.

\bibitem{silv}
Joseph~H. Silverman.
\newblock {\em The arithmetic of dynamical systems}, volume 241 of {\em
  Graduate Texts in Mathematics}.
\newblock Springer, New York, 2007.

\bibitem{stoll}
Michael Stoll.
\newblock Galois groups over {$\mathbf Q$} of some iterated polynomials.
\newblock {\em Arch. Math. (Basel)}, 59(3):239--244, 1992.

\end{thebibliography}

\end{document}